\documentclass[a4paper]{amsart}
\usepackage{amsmath,amsthm,amssymb,latexsym,epic,bbm,comment}
\usepackage{graphicx,enumerate,stmaryrd}
\usepackage[all,2cell]{xy}
\xyoption{2cell}

\usepackage[notref,notcite]{showkeys}
\usepackage[active]{srcltx}
\usepackage[parfill]{parskip}
\UseTwocells

\newtheorem{theorem}{Theorem}

\newtheorem{corollary}[theorem]{Corollary}
\newtheorem{proposition}[theorem]{Proposition}
\newtheorem{example}[theorem]{Example}

\newcommand{\tto}{\twoheadrightarrow}

\font\sc=rsfs10
\newcommand{\cC}{\sc\mbox{C}\hspace{1.0pt}}
\newcommand{\cG}{\sc\mbox{G}\hspace{1.0pt}}

\newcommand{\cR}{\sc\mbox{R}\hspace{1.0pt}}

\newcommand{\cS}{\sc\mbox{S}\hspace{1.0pt}}

\font\scc=rsfs7
\newcommand{\ccC}{\scc\mbox{C}\hspace{1.0pt}}

\newcommand{\ccS}{\scc\mbox{S}\hspace{1.0pt}}

\begin{document}

\title[Multisemigroups with multiplicities]{Multisemigroups with multiplicities\\ 
and complete ordered semi-rings}

\author{Love Forsberg}
\date{}

\begin{abstract}
Motivated by appearance of multisemigroups in the study of additive $2$-categories,
we define and investigate the notion of a multisemigroup with multiplicities. 
This notion seems to be better suitable for applications in higher representation theory.
\end{abstract}
\maketitle

\section{Introduction}\label{s1}

Abstract $2$-representation theory originates from the papers \cite{BFK,Kh,CR} and is nowadays 
formulated as study of $2$-representations of additive $\Bbbk$-linear $2$-categories, 
where $\Bbbk$ is the base field, see e.g. \cite{Ma} for details. Various aspects of 
general $2$-representation theory of abstract additive $\Bbbk$-linear $2$-categories were studied 
in the series \cite{MM1,MM2,MM3,MM4,MM5,MM6} of papers by Mazorchuk and Miemietz. An important role 
in this study is played by the so-called {\em multisemigroup} of an 
additive $\Bbbk$-linear $2$-category which was originally introduced in \cite{MM2}. 

Recall that a {\em multisemigroup} is a set $S$ endowed with a {\em multioperation},
that is a map $*:S\times S\to 2^S$ which satisfies the following {\em associativity} axiom:
\begin{displaymath}
\bigcup_{s\in a*b}s*c=\bigcup_{t\in b*c}a*t
\end{displaymath}
for all $a,b,c\in S$ (see \cite{KM15} for more details and examples). 
The original observation in \cite{MM2} is that the set $\mathcal{S}$
of isomorphism classes of indecomposable $1$-morphisms in an additive $\Bbbk$-linear $2$-category 
$\cC$ has the natural structure of a multisemigroup, given as follows: for two indecomposable
$1$-morphisms $\mathrm{F}$ and $\mathrm{G}$, we have
\begin{displaymath}
[\mathrm{F}]*[\mathrm{G}]=\{[\mathrm{H}]\,:\,
\mathrm{H} \text{ is isomorphic to a direct summand of }\mathrm{F}\circ\mathrm{G}\},
\end{displaymath}
where $[\mathrm{F}]$ stands for the isomorphism class of $\mathrm{F}$ and
$\circ$ denotes composition in $\cC$. We refer the reader to \cite{MM2} for details.
Combinatorics of this multisemigroup reflects and encodes various structural properties of the 
underlying additive $\Bbbk$-linear $2$-category and controls major parts of the  $2$-representation 
theory of the latter, see \cite{MM1,MM2,MM3,MM4,MM5,MM6} for details.

However, this notion of a multisemigroup of an additive $\Bbbk$-linear $2$-category has one disadvantage:
it seems to forget too much information. In more details, it only records information about 
direct summands appearing in the composition $\mathrm{F}\circ\mathrm{G}$, however, it forgets 
information about {\em multiplicities} with which these direct summands appear. As as result,
the multisemigroup of an additive $\Bbbk$-linear $2$-category can not be directly applied 
to the study of the split Grothendieck category of $\cC$ and linear representations of the latter.

It is quite clear how one can amend the situation: one has to define a weaker notion than a multisemigroup
which should keep track of multiplicities in question. This naturally leads to the notion of 
{\em multisemigroups with multiplicities}, or {\em multi-multisemigroups} which is the object of the
study in this paper (the idea of such object is mentioned in \cite[Remark~8]{MM2} without any details). 
Although the definition is rather obvious under natural finiteness assumption, the
full generality setup has some catches and thus requires some work. The main aim of the present paper
is to analyze this situation and to propose a ``correct'' definition of a multi-multisemigroup.
The main value of the paper lies not in the difficulty of the presented results but rather in
the thorough analysis of the situation which explores various connections of the theory we initiate. 
Our approach utilizes the algebraic theory of complete semirings.

The paper is organized as follows: in Section~\ref{s2} we outline in more details the motivation for
this study coming from the higher representation theory. In Section~\ref{s3} we collect all notions 
and tool necessary to define our main object: multi-multisemigroups, or, how we also call them,
multisemigroups with multiplicities bounded by some cardinal.
Section~\ref{s4} ties back to the original motivation and is devoted to the analysis of 
multisemigroups with multiplicities appearing in the higher representation theory.
In Section~\ref{s5} we give some explicit examples. In Section~\ref{s6} we discuss 
multi-multisemigroups for different sets of multiplicities and 
connection to twisted semigroup algebras.
Finally, in Section~\ref{s7}, we describe multi-multisemigroups as algebras
over complete semirings.

\section{Motivation from the $2$-representation theory}\label{s2}

\subsection{$2$-categories}\label{s2.1}

For details on $2$-categories we refer the reader to \cite{Le,Ma}. 

A {\em $2$-category} is a category enriched over the category of small categories. In other words,
a $2$-category $\cC$ consists of 
\begin{itemize}
\item objects,
\item small categories of morphisms,
\item identity objects in the appropriate
morphisms categories,
\item bifunctorial composition
\end{itemize}
and all these data is supposed to satisfy all the obvious axioms.
The canonical example of a $2$-category is the category $\mathbf{Cat}$ of small categories where 
\begin{itemize}
\item objects are small categories,
\item morphisms are categories where objects are functors and morphisms are natural
transformations of functors,
\item identities are the identity functors,
\item composition is composition of functors. 
\end{itemize}
As usual, if $\cC$ is a $2$-category and $\mathtt{i},\mathtt{j}\in\cC$, the objects in 
$\cC(\mathtt{i},\mathtt{j})$ are called {\em  $1$-morphisms} and morphisms in 
$\cC(\mathtt{i},\mathtt{j})$ are called {\em  $2$-morphisms}. Composition of 
$2$-morphisms inside $\cC(\mathtt{i},\mathtt{j})$ is called {\em vertical}, while
composition of $2$-morphisms in $\cC$ is called {\em horizontal}.

Let $\Bbbk$ be a field. We will say that a $2$-category $\cC$ is {\em $\Bbbk$-admissible} provided that, 
\begin{itemize}
\item for any $\mathtt{i},\mathtt{j}\in\cC$, the category $\cC(\mathtt{i},\mathtt{j})$
is $\Bbbk$-linear, idempotent split and Krull-Schmidt,
\item composition is $\Bbbk$-bilinear.
\end{itemize}
For example, let $A$ be a finite-dimensional associative algebra and $\mathcal{C}$ a small category
equivalent to $A$-mod, then the $2$-full subcategory  $\cR_{(A,\mathcal{C})}$ of
$\mathbf{Cat}$ with unique object $\mathcal{C}$ and whose $1$-morphisms are right exact 
endofunctors on $\mathbb{C}$, is $\Bbbk$-admissible. The reason for this is the fact that 
$\cR_{(A,\mathcal{C})}(\mathcal{C},\mathcal{C})$ is equivalent to the category of
$A$-$A$--bimodules, see \cite{Ba} for details.

\subsection{Grothendieck category of a $\Bbbk$-admissible $2$-category}\label{s2.2}

Let $\mathcal{C}$ be an additive category. Then the {\em split Grothendieck group} 
$[\mathcal{C}]_{\oplus}$ of $\mathcal{C}$ is defined as the quotient of the free abelian group
generated by $[X]$, where $X\in \mathcal{C}$, modulo the relations $[X]+[Y]-[Z]$ whenever
$Z\cong X\oplus Y$. If $\mathcal{C}$ is idempotent split and Krull-Schmidt, then 
$[\mathcal{C}]_{\oplus}$ is isomorphic to the free abelian group generated by
isomorphism classes of indecomposable objects in $\mathcal{C}$.

Let $\cC$ be a $\Bbbk$-admissible $2$-category.  The associated {\em Grothendieck category}
$[\cC]_{\oplus}$, also called the {\em decategorification} of $\cC$, is defined as
the category such that
\begin{itemize}
\item $[\cC]_{\oplus}$ has the same objects as $\cC$,
\item for $\mathtt{i},\mathtt{j}\in[\cC]_{\oplus}$, we have  $[\cC]_{\oplus}(\mathtt{i},\mathtt{j}):=
[\cC(\mathtt{i},\mathtt{j})]_{\oplus}$,
\item identity morphisms in $[\cC]_{\oplus}$ are classes of the corresponding identity $1$-morphisms in $\cC$,
\item composition in $[\cC]_{\oplus}$ is induced from the composition in $\cC$.
\end{itemize}
We note that the category $[\cC]_{\oplus}$ is, by definition, preadditive, but not additive in general
(as, in general, no coproduct of objects in $\cC$ was assumed to exist).

\begin{example}\label{ex1}
{\em 
Let $S$ be a finite semigroup with an admissible partial order $\leq$. 
Define the $2$-category $\cS$ as follows:
\begin{itemize}
\item $\cS_S$ has one object $\mathtt{i}$;
\item $1$-morphisms in $\cS_S$ are elements from $S$;
\item composition of $1$-morphisms is given by multiplication in $S$;
\item for two $1$-morphisms $s,t\in S$, we have 
\begin{displaymath}
\mathrm{Hom}_{\ccS}(s,t):=
\begin{cases}
\varnothing,& s\not \leq t;\\
\{\mathbf{h}_{s,t}\},& s\leq t.
\end{cases}
\end{displaymath}
\item vertical composition of $2$-morphism is defined in the unique possible way which 
is justified by transitivity of $<$;
\item horizontal composition of $2$-morphism is defined in the unique possible
way, which is justified by admissibility of $<$.
\end{itemize}
For a field $\Bbbk$, define the $\Bbbk$-linearization $\cS_{\Bbbk}$ of $\cS$ as follows, see 
\cite[Subsection~4.3]{GM} for details:
\begin{itemize}
\item $\cS_\Bbbk$ has one object $\mathtt{i}$;
\item $1$-morphisms in $\cS_\Bbbk$ are formal finite direct sums of $1$-morphisms in $\cS$;
\item $2$-morphisms in $\cS_\Bbbk$ are appropriate matrices whose entries are 
in $\Bbbk\mathbf{h}_{s,t}$;
\item compositions in $\cS_\Bbbk$ are induced from those in $\cS$ using $\Bbbk$-bilinearity.
\end{itemize}
The $2$-category $\cS_\Bbbk$ is, by construction, $\Bbbk$-admissible. Moreover, the decategorification
$[\cS_\Bbbk]_{\oplus}$ of this $2$-category 
\begin{itemize}
\item has one object $\mathtt{i}$;
\item the endomorphism ring $[\cS_\Bbbk]_{\oplus}(\mathtt{i},\mathtt{i})$ of the object $\mathtt{i}$ 
is isomorphic to the integral semigroup ring $\mathbb{Z}[S]$. 
\end{itemize}

}
\end{example}

\subsection{Finitary $2$-categories}\label{s2.3}

A $\Bbbk$-admissible $2$-category  $\cC$ is called {\em finitary}, see \cite{MM1}, provided that
\begin{itemize}
\item it has finitely many objects;
\item it has finitely many indecomposable $1$-morphism, up to isomorphism;
\item all $\Bbbk$-spaces of $2$-morphisms are finite dimensional;
\item all identity $1$-morphisms are indecomposable.
\end{itemize}
For example, the category $\cS_\Bbbk$ constructed in Example~\ref{ex1} is finitary (by construction
and using the fact that $S$ is finite).

\subsection{Multisemigroup of a $\Bbbk$-admissible $2$-category}\label{s2.4}

Let $\cC$ be a $\Bbbk$-admissible $2$-category. Consider the set $\mathcal{S}(\cC)$
of isomorphism classes of indecomposable $1$-morphisms in $\cC$. Recall,
from Section~\ref{s1}, that setting, for two indecomposable $1$-morphisms $\mathrm{F}$ 
and $\mathrm{G}$ in $\cC$,
\begin{equation}\label{eq1}
[\mathrm{F}]*[\mathrm{G}]=\{[\mathrm{H}]\,:\,
\mathrm{H} \text{ is isomorphic to a direct summand of }\mathrm{F}\circ\mathrm{G}\},
\end{equation}
defines on $\mathcal{S}(\cC)$ the structure of a multisemigroup. For example, 
for the category $\cS_\Bbbk$ constructed in Example~\ref{ex1}, the multisemigroup
$\mathcal{S}(\cC)$ is canonically isomorphic to the semigroup $S$ (by sending
$[s]$ to $s$, for $s\in S$). In particular, in this case the multioperation defined
by \eqref{eq1} is, in fact, single-valued and thus the prefix ``multi'' is redundant.

\begin{example}\label{ex2}
{\rm 
Consider the symmetric group $S_3$ as a Coxeter group with generators $s$
(standing for the elementary transposition $(1,2)$) and $t$
(standing for the elementary transposition $(2,3)$). Then
\begin{displaymath}
S_3:=\{e,s,t,st,ts,sts\}, 
\end{displaymath}
where $s^2=t^2=e$ and $sts=tst$. Then we have the following {\em Kazhdan-Lusztig basis} in
$\mathbb{Z}[S_3]$:
\begin{gather*}
\overline{e}:=e,\quad \overline{s}:=e+s,\quad \overline{t}:=e+t, \quad
\overline{st}:=e+s+t+st,\\
\overline{ts}:=e+s+t+ts,\quad \overline{sts}:=e+s+t+ts+st+sts.
\end{gather*}
The multiplication table of the Kazhdan-Lusztig basis elements is given by:
\begin{equation}\label{eq2}
\begin{array}{c||c|c|c|c|c|c}
\cdot& \overline{e} & \overline{s} & \overline{t} & \overline{st} & \overline{ts} & \overline{sts} \\
\hline\hline
\overline{e} & \overline{e} & \overline{s} & \overline{t} & \overline{st} & \overline{ts} & \overline{sts} \\
\hline
\overline{s} & \overline{s} & 2\overline{s} & \overline{st} & 2\overline{st} & 
\overline{sts}+\overline{s} & 2\overline{sts} \\
\hline
\overline{t} & \overline{t} & \overline{ts} & 2\overline{t} & \overline{tst}+\overline{t} & 
2\overline{ts} & 2\overline{sts} \\
\hline
\overline{st} & \overline{st} & \overline{sts}+\overline{s} & 2\overline{st} & 2\overline{tst}+\overline{st} & 
2\overline{sts}+2\overline{s} & 4\overline{sts} \\
\hline
\overline{ts} & \overline{ts} & 2\overline{ts} & \overline{tst}+\overline{t} & 2\overline{tst}+2\overline{t} & 
2\overline{sts}+\overline{ts} & 4\overline{sts} \\
\hline
\overline{sts} & \overline{sts} & 2\overline{sts} & 2\overline{sts} & 4\overline{sts} & 
4\overline{sts} & 6\overline{sts} 
\end{array}
\end{equation}
Consider the $2$-category $\cS_3$ of Soergel bimodules over the coinvariant algebra of $S_3$
as detailed in, e.g., \cite[Subsection~7.1]{MM1}. Consider the corresponding 
Grothendieck category $[\cS_3]_{\oplus}$. Then the ring $[\cS_3]_{\oplus}(\mathtt{i},\mathtt{i})$
is isomorphic to $\mathbb{Z}[S_3]$ where the isomorphism sends isomorphism classes of indecomposable
$1$-morphisms in $\cS_3$ to elements of the Kazhdan-Lusztig basis. This means that 
$\mathcal{S}[\cS_3]$ can be identified with $S_3$ as a set. From \eqref{eq2} it follows that the
multioperation $*$ on $\mathcal{S}[\cS_3]$ is given by:
\begin{equation}\label{eq3}
\begin{array}{c||c|c|c|c|c|c}
\cdot& {e} & {s} & {t} & {st} & {ts} & {sts} \\
\hline\hline
{e} & \{e\} & \{s\} & \{t\} & \{st\} & \{ts\} & \{sts\} \\
\hline
{s} & \{s\} & \{s\} & \{st\} & \{st\} & 
\{sts,s\} & \{sts\} \\
\hline
{t} & \{t\} & \{ts\} & \{t\} & \{tst,t\} & 
\{ts\} & \{sts\} \\
\hline
{st} & \{st\} & \{sts,s\} & \{st\} & \{tst,st\} & 
\{sts,s\} & \{sts\} \\
\hline
{ts} & \{ts\} & \{ts\} & \{tst,t\} & \{tst,t\} & 
\{sts,ts\} & \{sts\} \\
\hline
{sts} & \{sts\} & \{sts\} & \{sts\} & \{sts\} & 
\{sts\} & \{sts\} 
\end{array} 
\end{equation}
Here we see that the multioperation $*$ is genuinely multi-valued.
} 
\end{example}

\subsection{Multisemigroup and decategorification}\label{s2.5}

Comparing \eqref{eq2} with \eqref{eq3}, it is easy to see that the information encoded
by the multisemigroup, that is \eqref{eq3}, is not enough to recover the 
``associative algebra structure'' which exists on the level of the Grothendieck decategorification
presented in \eqref{eq2}. The essential part of the information which got lost is
the exact values of non-zero multiplicities with which indecomposable $1$-morphism
appear in composition of two given indecomposable $1$-morphisms.

One can say that the situation is even worse. Let us try to use \eqref{eq3} to define
{\em some} associative algebra structure on the abelian group $\mathbb{Z}[S_2]$. The only
reasonable guess would be to define, on generators, an operation $\diamond$ as follows: 
\begin{displaymath}
x\diamond y=\sum_{z\in x*y}z
\end{displaymath}
and then extend this to $\mathbb{Z}[S_2]$ by bilinearity. However, this is not associative,
for example, $(sts\diamond st)\diamond s\neq sts\diamond(st\diamond s)$, indeed,
\begin{displaymath}
(sts\diamond st)\diamond s=sts\diamond s=sts,\qquad sts\diamond (st\diamond s)=sts\diamond (sts+s)=2sts.
\end{displaymath}
To have associativity, we should have considered $\mathbf{B}[S_2]$, where 
$\mathbf{B}$ is the Boolean semiring. This will be explained in details later.

Therefore, if we want to define some discrete object which we could use to recover 
the associative algebra structure given by the Grothendieck decategorification,
we need to keep track of multiplicities. This naturally leads to the 
notion of multisemigroups with multiplicities.

\section{Multisemigroups with multiplicities}\label{s3}

\subsection{Semirings}\label{s3.1}

A semiring is a weaker notion than that of a ring and the difference is that it is 
only required to form a commutative monoid (not a group) with respect to addition.
More precisely, a {\em unital semiring} is a tuple $(R,+,\cdot,0,1)$, where
\begin{itemize}
\item $R$ is a set;
\item $+$ and $\cdot$ are binary operations on $R$;
\item $0$ and $1$ are two different elements of $R$.
\end{itemize}
These data is required to satisfy the following axioms.
\begin{itemize}
\item $(R,+,0)$ is a commutative monoid with identity element $0$;
\item $(R,\cdot,1)$ is a monoid with identity element $1$;
\item multiplication distributes over addition both from the left and from the right;
\item $0\cdot R= R\cdot 0=0$.
\end{itemize}
We refer to \cite{Go,Ka} for more details.

Here are some examples of semirings:
\begin{itemize}
\item Any unital ring is a unital semiring.
\item $\mathbb{Z}_{\geq 0}=(\{0,1,2,3,\dots\},+,\cdot,0,1)$.
\item The Boolean semiring $\mathbf{B}=(\{0,1\},+,\cdot,0,1)$ with respect to the 
usual boolean addition and multiplication given by:
\begin{displaymath}
\begin{array}{c||c|c}
+&0&1\\
\hline\hline
0&0&1\\
\hline
1&1&1
\end{array}\qquad\text{ and }\quad
\begin{array}{c||c|c}
\cdot&0&1\\
\hline\hline
0&0&0\\
\hline
1&0&1
\end{array}
\end{displaymath}
\item The dual Boolean semiring $\mathbf{B}^{*}=(\{0,1\},\cdot,+,1,0)$ with respect to the 
boolean multiplication (as addition) and boolean addition (as multiplication).
\item If $R$ is a semiring, then the set $\mathrm{Mat}_{n\times n}(R)$ of $n\times n$ matrices 
with coefficients in $R$ forms a semiring with respect to the usual addition and multiplication of
matrices.
\item For any nonempty set $X$, we have the semiring  $\mathbf{B}_X:=(\mathbf{B}^{X},\cup,\cap,\varnothing,X)$.
This semiring is isomorphic to 
\begin{displaymath}
\prod_{x\in X} \mathbf{B}^{(x)}, 
\end{displaymath}
where $\mathbf{B}^{(x)}=\mathbf{B}$, a copy of the Boolean semiring $\mathbf{B}$ indexed by $x$.
\end{itemize}

Given two semirings $R=(R,+,\cdot,0,1)$ and $R'=(R',+',\cdot',0',1')$, a homomorphism 
$\varphi:R\to R'$ is a map from $R$ to $R'$ such that
\begin{itemize}
\item $\varphi(r+s)=\varphi(r)+'\varphi(s)$, for all $r,s\in R$;
\item $\varphi(r\cdot s)=\varphi(r)\cdot'\varphi(s)$, for all $r,s\in R$;
\item $\varphi(0)=0'$;
\item $\varphi(1)=1'$.
\end{itemize}
Semirings and homomorphisms form a category, denoted by $\mathbf{SRing}$.

\subsection{Complete semirings}\label{s3.2}

A commutative monoid $(S,+,0)$ is called {\em complete} provided that it is equipped,
for any indexing set $I$, with the sum operation $\displaystyle\sum_{i\in I}$ such that
\begin{itemize}
\item $\displaystyle \sum_{i \in \varnothing}{r_i} =0$;
\item $\displaystyle \sum_{i \in \{j\}}{r_i} = r_j$;
\item $\displaystyle \sum_{j \in J}{\sum_{i \in I_j}{r_i}} = \sum_{s \in I}r_s$ when
$\displaystyle \bigcup_{j\in J} I_j=I$ and $I_j \cap I_{j'} = \varnothing$ for $j\neq j'$.
\end{itemize}
We refer the reader to \cite{He} for more details.

A semiring $(R,+,\cdot,0,1)$ is called {\em complete} provided that
\begin{itemize}
\item $(R,+,0)$ is a complete monoid;
\item multiplication distributes over all operations $\displaystyle\sum_{i\in I}$ on both sides, that is
\begin{displaymath}
r\cdot\left(\sum_{i\in I}r_i\right)= \sum_{i\in I}(r\cdot r_i)\qquad\text{ and }\qquad
\left(\sum_{i\in I}r_i\right)\cdot r= \sum_{i\in I}(r_i\cdot r). 
\end{displaymath}
\end{itemize}

Given two complete semirings 
\begin{displaymath}
R=\left(R,+,\cdot,0,1,\sum_{i\in I}\right)\quad \text{ and }\quad 
R'=\left(R',+',\cdot',0',1',\sum'_{i\in I}\right), 
\end{displaymath}
a homomorphism $\varphi:R\to R'$ is a homomorphism of underlying semiring such that
\begin{displaymath}
\varphi\left(\sum_{i\in I}r_i\right)= \sum'_{i\in I}\varphi(r_i),\quad\text{ for all }\quad
r_i\in R. 
\end{displaymath}
Complete semirings and homomorphisms form a subcategory in $\mathbf{SRing}$, denoted by $\mathbf{CSRing}$.

Here are some examples of complete semirings:
\begin{itemize}
\item Any bounded complete join-semilattice is a complete commutative semi-ring.
\item $(\mathbf{B}^X,\cup,\cap,\varnothing,X)$, for some set $X$, where $\sum_{i\in I}$ is the usual union.
\item The set of open sets for a topological space $X$, with respect to union and intersection.
\item Unital quantales with join as addition and the underlying associative operation as multiplication.
\item Integral tropical semiring $(\mathbb{Z}_{\geq 0}\cup\{\infty\},\mathrm{max},+,\infty,0)$, where
$\sum_{i\in I}$ is just taking the supremum.
\item The semiring $(\mathbb{Z}_{\geq 0}\cup\{\infty\},+,\mathrm{min},0,\infty)$, where
the sum of infinitely many nonzero elements is set to be $\infty$.
\item The semiring $(\mathbb{R}_{\geq 0}\cup\{\infty\},+,\cdot,0,1)$, where
$\sum_{i\in I}$ is defined as the supremum over all finite partial subsums.
\item The semiring $(\mathbb{R}_{\geq 0}\cup\{\infty\},+,\cdot,0,1)$, where
any infinite sum of non-zero elements is defined to be $\infty$.
\end{itemize}

It is very tempting to add to the above the following ``example'': all cardinal numbers form a complete 
semiring with respect to the usual addition (disjoint union) and  multiplication (Cartesian product) of cardinals.
There is one problem with this ``example'', namely, 
the fact that cardinals do not form a set but, rather, a proper class.
This problem can be overcome in an artificial and non-canonical way described in the next example.
This examples is separated from the rest due to  its importance in what follows.

\begin{example}\label{ex3}
{\rm 
For a fixed cardinal $\kappa$, let $\mathrm{Card}_{\kappa}$ denote the set of all cardinals which 
are not greater than $\kappa$. Then $\mathrm{Card}_{\kappa}$ has the structure of a complete 
semiring where
\begin{itemize}
\item addition (of any number of elements) is given by disjoint union with convention 
that all cardinals greater than $\kappa$ are identified with $\kappa$;
\item multiplication is given by Cartesian product with convention that all 
cardinals greater than $\kappa$ are identified with $\kappa$.
\end{itemize}
}
\end{example}

Note that the Boolean semiring  $\mathbf{B}$ is isomorphic to $\mathrm{Card}_{1}$.

\subsection{Multisets}\label{s3.3}

Recall, see e.g. \cite[Page~1]{Au}, that a   {\em classical multiset} is a pair $(A,\mu)$, where
\begin{itemize}
\item $A$ is a set;
\item $\mu:A\to \mathbb{Z}_{\geq 0}$ is a function, called the {\em multiplicity function}.
\end{itemize}
A natural, more general, notion is that of a genuine {\em multiset}, which is a pair $(A,\mu)$, where
\begin{itemize}
\item $A$ is a set;
\item $\mu$ is the multiplicity function from $A$ to the class of all cardinals.
\end{itemize}

\subsection{Multi-Booleans}\label{s3.4}

Recall that, given a base set $X$, the {\em Boolean} $\mathcal{B}(X)=\mathbf{B}^{X}$ of $X$ is the set of all
subsets of $X$. This can be identified with the set of all functions from $X$ to the Boolean semiring
$\mathbf{B}$. In this way, $\mathcal{B}(X)$ gets the natural structure of a complete semiring
with respect to the union and intersection of subsets. The additive identity is the empty subset while
the multiplicative identity is $X$. Note that we can also consider the {\em dual Boolean} 
of $X$ which is  the set of all functions from $X$ to the dual Boolean semiring
$\mathbf{B}^*$. This gets the natural structure of a complete semiring
with respect to the intersection and union of subsets. The additive identity is $X$ while
the multiplicative identity is the empty subset.

The above point of view allows us to generalize the definition of the Boolean to multiset structures.
Given a base set $X$, define the {\em full multi-Boolean}  of $X$ is the class
of all functions from $X$ to the class of all cardinal numbers. To create any sensible theory,
we need sets. This motivates the following definition.

Given a base set $X$ and a cardinal number $\kappa\geq$, define the {\em $\kappa$-multi-Boolean} 
$\mathcal{B}_{\kappa}(X)$ of $X$ is the set of all functions from $X$ to 
the complete semiring $\mathrm{Card}_{\kappa}$. By construction, 
$\mathcal{B}_{\kappa}(X)$ is equipped with the natural structure of a complete semiring.
Also, we have $\mathcal{B}(X)=\mathcal{B}_1(X)$. 

Clearly $\kappa=0$ would give us a singelton, on which no semi-ring structure exists. From now on we agree that any cardinal $\kappa$ in this paper is greater than or equal to 1.

Unfortunately, for $\kappa>1$, the natural complete semiring structure on $\mathcal{B}_{\kappa}(X)$
does not correspond to the usual set-theoretic notions of union and intersection of 
multisets. Indeed, the multiplicity analogue of the intersection of multisets is the arithmetic
operation of taking minimum, while the multiplicity analogue of the union of multisets is 
the arithmetic operation of taking maximum. These differ from the usual addition and multiplication
in $\mathrm{Card}_{\kappa}$, if $\kappa>1$.

\subsection{Multisemigroups with multiplicity}\label{s3.5}

Now we are ready to present our main definition. Let  $\kappa$ be a fixed cardinal. 
A {\em multisemigroup with multiplicities bounded by $\kappa$} is a pair
$(S,\mu)$, where 
\begin{itemize}
\item $S$ is a non-empty set;
\item $\mu:S\times S\to \mathcal{B}_{\kappa}(S)$, written $(s,t)\mapsto \mu_{s,t}:S\to \mathrm{Card}_{\kappa}$;
\end{itemize}
such that the following {\em distributivity} requirement is satisfied: for all $r,s,t\in R$, we have
\begin{equation}\label{eq5}
\sum_{i\in S}\mu_{s,t}(i)\mu_{r,i} =\sum_{j\in S}\mu_{r,s}(j)\mu_{j,t}.
\end{equation}
We note that here, for a cardinal $\lambda$ and a function $\nu:S\to \mathrm{Card}_{\kappa}$, by
$\lambda\nu$ we mean the function from $S$ to $\mathrm{Card}_{\kappa}$ defined as
\begin{displaymath}
\lambda \nu=\sum_{i\in\lambda}\nu,
\end{displaymath}
or, in other words, this is just adding up $\lambda$ copies of $\nu$.

The informal explanation for the requirement \eqref{eq5} is as follows: the left hand side corresponds
to the ``product'' $r*(s*t)$. Here $s*t$ gives $\mu_{s,t}$, which counts every $i\in S$ with multiplicity
$\mu_{s,t}(i)$. The result of $r*(s*t)$, written when we distribute $r*$ over all such $i\in S$ and taking
multiplicities into account, gives exactly the left hand side in \eqref{eq5}. Similarly, the right hand
side corresponds to the ``product'' $(r*s)*t$.

If $\kappa$ is clear from the context, we will sometimes use the shorthand 
{\em multi-multi\-semi\-group} instead of ``multisemigroup with multiplicities bounded by $\kappa$''.

Here are some easy examples of multisemigroup with multiplicities:
\begin{itemize}
\item A usual multisemigroup is a multisemigroup with multiplicities bounded by one.
\item For any $\kappa$ and any $\lambda<\kappa$, the set $\{a\}$ has the 
structure of a multisemigroup with multiplicities bounded by $\kappa$, if we
set $\mu_{a,a}=\lambda$. Moreover, these exhaust all such structures on $\{a\}$.
\end{itemize}

Here is a more involved example:

\begin{example}\label{ex4}
{\rm  
Let $A$ be a finite dimensional $\mathbb{R}$-algebra with a fixed basis
$\{a_i\,:\,i\in I\}$ such that $\displaystyle a_i\cdot a_j=\sum_{s\in I} \mu^s_{i,j}a_s$
and all $\mu^s_{i,j}\in\mathbb{Z}_{\geq 0}$. Then $(I,\mu)$, where we define
$\mu_{i,j}(s):=\mu^s_{i,j}$, is a multisemigroup with multiplicities bounded by $\omega$,
the first infinite cardinal. This follows from the associativity of multiplication in 
$A$ via the computation
\begin{displaymath}
\sum_{s}\sum_t \mu^t_{i,j}\mu^s_{t,k}a_s =
(a_i\cdot a_j)\cdot a_k=
a_i\cdot (a_j\cdot a_k)=
\sum_{a}\sum_b \mu^a_{i,b}\mu^b_{j,k}a_a,
\end{displaymath}
which is equivalent to \eqref{eq5} in this case since basis elements are linearly independent.
}
\end{example}

Let $(S,\mu)$ and $(S',\mu')$ be two multisemigroups with multiplicities bounded by $\kappa$. We
will say that they are {\em isomorphic} provided that there is a bijection $\varphi:S\to S'$
such that $\mu'=\varphi\circ\mu$.

Let $(S,\mu)$ be a multisemigroups with multiplicities bounded by $\kappa$.
Let $S^\diamond$ denote the set of all words in the alphabet $S$ of length at least two. 
Define the map 
\begin{displaymath}
\overline{\mu}:S^\diamond\to \mathcal{B}_{\kappa}(S)
\end{displaymath}
recursively as follows:
\begin{enumerate}
\item $\overline{\mu}_{st}=\mu_{s,t}$, if $s,t\in S$;
\item if $w=sx$, where $x$ has length at least two, then set
\begin{equation}\label{eq5-1}
\overline{\mu}_w(t):=\sum_{a\in S} \overline{\mu}_x(a)\mu_{s,a}(t).
\end{equation}
\end{enumerate}

The definition of $\overline{\mu}$ does not really depend on our choice of prefix above
(in contrast to suffix), as follows from the following statement.

\begin{proposition}\label{prop17}
If $w\in S^\diamond$ has the form $w=xs$, where $x$ has length at least two, then 
\begin{equation}\label{eq5-2}
\overline{\mu}_w(t):=\sum_{a\in S} \overline{\mu}_x(a)\mu_{a,s}(t).
\end{equation} 
\end{proposition}

\begin{proof}
Let $w=s_1s_2\dots s_k$, where $k\geq 3$. Then the recursive procedure in
\eqref{eq5-1} results in
\begin{equation}\label{eq5-3}
\sum_{i_1\in S}\sum_{i_2\in S}\cdots \sum_{i_{k-2}\in S} 
\mu_{s_{1},i_{1}}(t)\mu_{s_{2},i_{2}}(i_{1})
\cdots\mu_{s_{k-2},i_{k-2}}(i_{k-3})\mu_{s_{k-1},s_k}(i_{k-2}) .
\end{equation}
The recursive procedure in
\eqref{eq5-2} results in
\begin{equation}\label{eq5-4}
\sum_{j_1\in S}\sum_{j_2\in S}\cdots \sum_{j_{k-2}\in S} 
\mu_{s_{1},s_{2}}(j_1)\mu_{j_{1},s_{3}}(j_{2})
\cdots\mu_{j_{k-3},s_{k-1}}(j_{k-2})\mu_{j_{k-2},s_k}(t) .
\end{equation}
The expression \eqref{eq5-3} it transferred to \eqref{eq5-4} using a repetitive 
application of \eqref{eq5}. The claim follows. 
\end{proof}

\subsection{Finitary multisemigroups with multiplicities}\label{s3.6}

We will say that a multisemigroup $(S,\mu)$ with multiplicities bounded by $\kappa$
is {\em finitary} provided that 
\begin{itemize}
\item $\kappa=\aleph_0$;
\item $\mu_{r,s}(t)\neq \aleph_0$ for all $r,s,t\in S$;
\item $|\{t\in S\,:\, \mu_{r,s}(t)\neq 0\}|<\aleph_0$ for all $r,s\in S$. 
\end{itemize}

\subsection{Multi-multisemigroup of a $\Bbbk$-admissible $2$-category}\label{s3.7}

Let $\cC$ be a $\Bbbk$-ad\-mis\-si\-ble $2$-category. Consider the set $\mathcal{S}(\cC)$
of isomorphism classes of indecomposable $1$-morphisms in $\cC$. For $F,G,H\in \mathcal{S}(\cC)$,
define $\mu_{F,G}(H)$ to be the multiplicity of $H$ as a direct summand in the composition $F\circ G$.

\begin{theorem}\label{prop7}
The construct $(\mathcal{S}(\cC),\mu)$ is a finitary multisemigroup with multiplicities. 
\end{theorem}

\begin{proof}
We only have to check \eqref{eq5} in this case, as the rest follows by construction from 
$\Bbbk$-admissibility of $\cC$. For $F,G,H,K\in \mathcal{S}(\cC)$, the multiplicity of
$K$ in $(F\circ G)\circ H$ is given by 
\begin{displaymath}
\sum_{Q\in \mathcal{S}(\ccC)} \mu_{F,G}(Q)\mu_{Q,H}(K).
\end{displaymath}
In turn, the multiplicity of $K$ in $F\circ (G\circ H)$ is given by 
\begin{displaymath}
\sum_{P\in \mathcal{S}(\ccC)} \mu_{F,P}(K)\mu_{G,H}(P).
\end{displaymath}
As $(F\circ G)\circ H\cong F\circ (G\circ H)$ and $\mathcal{S}(\cC)$ is Krull-Schmidt, we have
\begin{displaymath}
\sum_{Q\in \mathcal{S}(\ccC)} \mu_{F,G}(Q)\mu_{Q,H}(K)=
\sum_{P\in \mathcal{S}(\ccC)} \mu_{F,P}(K)\mu_{G,H}(P),
\end{displaymath}
which proves \eqref{eq5} in this case.
\end{proof}

\begin{example}\label{ex8}
{\rm 
For the $2$-category $\cS_3$ in Example~\ref{ex2}, the multi-multisemigroup structure on
$\mathcal{S}(\cC)$ is fully determined by \eqref{eq2}. For instance, the function
$\mu_{\overline{st},\overline{st}}$ has the following values:
\begin{displaymath}
\begin{array}{c||c|c|c|c|c|c}
x:&\overline{e}&\overline{s}&\overline{t}&\overline{st}&\overline{ts}&\overline{sts}\\
\hline\hline
\mu_{\overline{st},\overline{st}}(x):&
0&0&0&1&0&2.
\end{array} 
\end{displaymath}
The function $\mu_{\overline{ts},\overline{sts}}$ has the following values:
\begin{displaymath}
\begin{array}{c||c|c|c|c|c|c}
x:&\overline{e}&\overline{s}&\overline{t}&\overline{st}&\overline{ts}&\overline{sts}\\
\hline\hline
\mu_{\overline{st},\overline{st}}(x):&
0&0&0&0&0&4.
\end{array} 
\end{displaymath}
} 
\end{example}

\section{Multi-multisemigroup vs multisemigroup and decategorification}\label{s4}

\subsection{Multi-multisemigroup vs multisemigroup}\label{s4.1}

Consider the canonical surjective semiring homomorphism 
$\Phi:\mathrm{Card}_{\omega}\tto \mathrm{Card}_{1}\cong \mathbf{B}$ defined by
\begin{displaymath}
\Phi(x)=
\begin{cases}
0,& x=0;\\
1,& \text{ otherwise}.
\end{cases}
\end{displaymath}
As usual, we identify subsets in a set $X$ with $\mathbf{B}^X$.
The following proposition says that the multi-multisemigroup of $\cC$ has enough information to
recover the multisemigroup of $\cC$.

\begin{proposition}\label{prop11}
Let $\cC$ be a $\Bbbk$-admissible $2$-category. Then, for any $[F],[G]\in\mathcal{S}(\cC)$, we have
\begin{displaymath}
[F]*[G]=\Phi\circ \mu_{F,G}. 
\end{displaymath}
\end{proposition}

\begin{proof}
This follows directly from the definitions. 
\end{proof}

\subsection{The algebra of a finitary multi-multisemigroup}\label{s4.2}

Let $(S,\mu)$ be a finitary multi-multisemigroup. For a fixed commutative unital ring $\Bbbk$, 
consider the free $\Bbbk$-module  $\Bbbk[S]$ with basis $S$. Define on $\Bbbk[S]$ a $\Bbbk$-bilinear 
binary multiplication $\cdot$ by setting, for $s,t\in S$,
\begin{equation}\label{eq9}
s\cdot t:=\sum_{r\in S} \mu_{s,t}(r)r. 
\end{equation}

\begin{proposition}\label{prop12}
The construct $(\Bbbk[S],\cdot)$ is an associative $\Bbbk$-algebra.
\end{proposition}

\begin{proof}
We need to show that $(r\cdot s)\cdot t=r\cdot (s\cdot t)$, for all $r,s,t\in S$. Using \eqref{eq9} 
and $\Bbbk$-bilinearity of $\cdot$, this reduces exactly to the axiom \eqref{eq5}.
\end{proof}

\subsection{Grothendieck ring of a $\Bbbk$-admissible $2$-category}\label{s4.3}

Let $\cC$ be a small $\Bbbk$-admissible $2$-category. The {\em Grothendieck ring} $Gr(\cC)$ of $\cC$
is defined as follows:
\begin{itemize}
\item elements of $Gr(\cC)$ are elements in the free abelian group generated by
isomorphism classes of indecomposable $1$-morphisms;
\item addition in $Gr(\cC)$ is the obvious addition inside the free abelian group;
\item multiplication in $Gr(\cC)$ is induced from composition in $\cC$ using biadditivity.
\end{itemize}
The ring $Gr(\cC)$ is unital if and only if $\cC$ has finitely many objects. Otherwise
it is {\em locally unital}, where local units correspond to (summands of) the identity 
$1$-morphisms in $\cC$.

An alternative way to look at $Gr(\cC)$ is to understand it as the ring associated 
with the preadditive category $[\cC]_{\oplus}$ in the obvious way. Conversely, 
$[\cC]_{\oplus}$ is the variation of the ring $Gr(\cC)$ which has several objects, cf. \cite{Mi}.

\subsection{Multi-multisemigroup vs decategorification}\label{s4.4}

Our main observation in this subsection is the following connection between the
multi-multisemigroup of a finitary $2$-category and the Grothendieck ring of this
category.

\begin{proposition}\label{prop14}
Let $\cC$ be a finitary $2$-category and $\Bbbk$ a field. Then there is a
canonical isomorphism of $\Bbbk$-algebras,
\begin{displaymath}
\Bbbk\otimes_{\mathbb{Z}}Gr(\cC)\cong  \Bbbk[\mathcal{S}(\cC)].
\end{displaymath}
\end{proposition}

\begin{proof}
We define the map $\psi: \Bbbk\otimes_{\mathbb{Z}}Gr(\cC)\to \Bbbk[\mathcal{S}(\cC)]$ as the
$\Bbbk$-linear extension of the map which sends an isomorphism class of indecomposable 
$1$-morphisms in $\cC$ to itself. This map is, clearly, bijective. Moreover, it is a homomorphism
of rings since, on both sides, the structure constants with respect to the $\Bbbk$-basis, consisting
of  isomorphism class of indecomposable  $1$-morphisms in $\cC$, are given by non-negative integers
$\mu_{F,G}(H)$ as defined in Subsection~\ref{s3.5}. The claim of the proposition follows.
\end{proof}

Altogether, for a finitary $2$-category $\cC$, we have the following picture
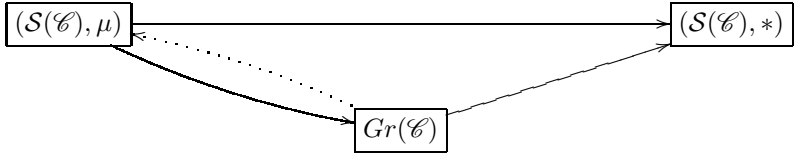
\begin{displaymath}
\xymatrix{ 
*+[F]{(\mathcal{S}(\cC),\mu)}\ar@/_/[rrrd]\ar[rrrrrr]&&&&&&*+[F]{(\mathcal{S}(\cC),*)}\\\ 
&&&*+[F]{Gr(\cC)}\ar[rrru]\ar@{..>}@/_/[lllu]&&& 
} 
\end{displaymath}
where arrow show in which direction we can recover information.

\section{Some explicit examples of multi-multisemigroups of finitary $2$-categories}\label{s5}

\subsection{Projective functors for finite dimensional algebras}\label{s5.1}

Let $\Bbbk$ be an algebraically closed field and $A$ a connected, basic, non-semi-simple, finite dimensional,
unital $\Bbbk$-algebra. Let $\mathcal{C}$ be a small category equivalent to $A$-mod. Following
\cite[Subsection~7.3]{MM1}, we define the $2$-category $\cC_A$ as a subcategory
in $\mathbf{Cat}$ (not full) such that:
\begin{itemize}
\item $\cC_A$ has one object $\mathtt{i}$, which we identify with $\mathcal{C}$;
\item $1$-morphisms in $\cC_A$ are functors isomorphic to direct sums of the identity functors
and functors of tensoring with projective $A$-bimodules;
\item $2$-morphisms in $\cC_A$ are natural transformations of functors.
\end{itemize}
Note that all $1$-morphisms in $\cC_A$ are, up to isomorphism, functors of tensoring with 
$A$-bimodules. For simplicity we will just use certain bimodules to denote the corresponding 
isomorphism classes of $1$-morphisms.

Let $1=e_1+e_2+\dots+e_k$ be a decomposition of $1\in A$ into a sum of primitive, pairwise
orthogonal idempotents. Then indecomposable $1$-morphisms in $\cC_A$ correspond to bimodule
\begin{displaymath}
\mathbbm{1}:=A,\qquad F_{i,j}:=Ae_i\otimes_{\Bbbk}e_jA,\,\,\text{ where }\,\, i,j=1,2,\dots,k. 
\end{displaymath}
The essential part of the composition in $\cC_A$ is given by
\begin{displaymath}
F_{i,j}\circ F_{i',j'}=F_{i,j'}^{\oplus \dim e_jAe_{i'}}, 
\end{displaymath}
as follows from the computation
\begin{displaymath}
Ae_i\otimes_{\Bbbk}e_jA \otimes_A Ae_i\otimes_{\Bbbk}e_jA\cong
Ae_i\otimes_{\Bbbk}e_{j'}A^{\oplus \dim e_jAe_{i'}}.
\end{displaymath}

The above implies that 
\begin{displaymath}
\mathcal{S}(\cC_A)=\{\mathbbm{1},F_{i,j}\,:\, i,j=1,2,\dots,k\} 
\end{displaymath}
and the multiplicity function defining the multi-multisemigroup structure on $\mathcal{S}(\cC_A)$
is given by 
\begin{displaymath}
\mu_{F,G}(H)=
\begin{cases}
1,& H=G\text{ and }F=\mathbbm{1};\\
1,& H=F\text{ and }G=\mathbbm{1};\\
\dim e_jAe_{i'},& F=F_{i,j}, G=F_{i',j'}, H=F_{i,j'};\\
0, & \text{otherwise}. 
\end{cases}
\end{displaymath}
Note also that, in this case, the multioperation in the multisemigroup
$(\mathcal{S}(\cC_A),*)$ is, at most, single valued. By adding, if necessary, an external 
element $0$, we can turn $(\mathcal{S}(\cC_A),*)$ into a genuine semigroup.

\subsection{Soergel bimodules for finite Coxeter groups}\label{s5.2}

Another prominent example of a finitary $2$-category is the finitary $2$-category
of Soergel bimodules. Let $W$ be a finite Coxeter group with a fixed geometric
representation. To these data, one associates the so-called $2$-category $\cS_W$ 
of Soergel bimodules over the coinvariant algebra of the geometric representation,
see \cite{So} and \cite[Subsection~7.1]{MM1}. This is a finitary $2$-category.
This $2$-category categorifies the integral group ring of $W$ in the sense that
there is an isomorphism between the ring $[\cC]_{\oplus}(\mathtt{i},\mathtt{i})$
and the ring $\mathbb{Z}[W]$ given in terms of the Kazhdan-Lusztig basis in 
$\mathbb{Z}[W]$, see \cite{KL}. Therefore the multi-multisemigroup $(\mathcal{S}(\cS_W),\mu)$
records exactly the information about the structure constants of the 
ring $\mathbb{Z}[W]$ with respect to the Kazhdan-Lusztig basis. 
As far as we know, there is no explicit combinatorial formula for such structure 
constants, however, they can be determined using a recursive algorithm.

In the special case of a Dihedral group $D_n$, where $n\geq 3$, 
\begin{displaymath}
W=D_n=\{s,t\,:\, s^2=t^2=(st)^n=e\},
\end{displaymath}
the Kazhdan-Lusztig basis has particularly simple form. Elements of the group $D_n$
can be listed as
\begin{displaymath}
D_n=\{e,s,t,st,ts,\dots,w_0\},
\end{displaymath}
where $w_0=stst\dots=tsts\dots$, where the length of both words is $n$. Then, for each
$w\in D_n$, the corresponding Kazhdan-Lusztig basis element $\underline{w}\in \mathbb{Z}[D_n]$
is the sum of $w$ with all elements of strictly smaller length. 

Let $\mathbf{l}:D_n\to\mathbb{Z}_{\geq0}$ be the length function with respect to 
generators $s$ and $t$. A direct calculation then shows that 
\begin{displaymath}
\underline{s}\cdot \underline{w}=
\begin{cases}
2 \underline{w}, & \mathbf{l}(sw)<\mathbf{l}(w);\\
\underline{sw}, & w=e\text{ or }w=t;\\
\underline{sw}+\underline{tw}, & \text{otherwise};
\end{cases}
\qquad
\underline{t}\cdot \underline{w}=
\begin{cases}
2 \underline{w}, & \mathbf{l}(tw)<\mathbf{l}(w);\\
\underline{tw}, & w=e\text{ or }w=s;\\
\underline{sw}+\underline{tw}, & \text{otherwise}.
\end{cases}
\end{displaymath}
This already shows that the multi-multisemigroup structure is non-trivial in the sense
that it is not reducible to a multisemigroup structure. The above formulae determine
the multiplicity functions $\mu_{s,w}$ and $\mu_{t,w}$. As any element in $D_n$ is a product
of $s$ and $t$, all remaining multiplicity functions can be determined inductively. 
However, we do not know of any explicit formulae. For $n=3$, the answer is given in \eqref{eq2}.
More information on the $D_n$ case can be found in \cite{El}.

\subsection{Catalan monoid}\label{s5.3}

Let $n$ be a positive integer.
Consider the path algebra $A=A_n$ over $\mathbb{C}$ of the quiver
\begin{displaymath}
\xymatrix{1\ar[r]&2\ar[r]&3\ar[r]&\dots\ar[r]&n}. 
\end{displaymath}
Let $\mathcal{C}$ be a small category equivalent to $A$-mod. Following
\cite{GM0}, we define the $2$-category $\cG_n$ as a subcategory
in $\mathbf{Cat}$ (not full) such that:
\begin{itemize}
\item $\cC_A$ has one object $\mathtt{i}$, which we identify with $\mathcal{C}$;
\item $1$-morphisms in $\cC_A$ are functors isomorphic to direct sums of 
functors of tensoring with subbimodules of  ${}_AA_A$;
\item $2$-morphisms in $\cC_A$ are natural transformations of functors.
\end{itemize}
The main result of  \cite{GM0} asserts that the multisemigroup $\mathcal{S}(\cG_n)$
(with added zero) is isomorphic to the Catalan monoid $\mathtt{C}_{n+1}$ of all 
order-preserving and order-decreasing transformation of a finite chain with $n+1$ elements.
In particular, the multisemigroup $\mathcal{S}(\cG_n)$ is a semigroup.

Moreover, in \cite{GM0} it is also shown that the composition of indecomposable
$1$-morphism in $\cG_n$ is indecomposable (or zero). This means that, in this case,
the multi-multisemigroup structure on $\mathcal{S}(\cG_n)$ coincides with the 
multisemigroup structure.

A similar phenomenon was observed in some other cases in \cite{Zh1,Zh2}.

\section{Multi-multisemigroups with different multiplicities}\label{s6}

\subsection{Cardinal reduction}\label{s6.1}

Let $\lambda<\kappa$ be two cardinal numbers. Then there is a canonical homomorphism
\begin{displaymath}
\Phi_{\lambda,\kappa}: \mathrm{Card}_{\kappa}\to  \mathrm{Card}_{\lambda}
\end{displaymath}
of complete semirings defined as follows:
\begin{displaymath}
\Phi_{\lambda,\kappa}(\nu)=
\begin{cases}
\nu,& \nu\leq \lambda;\\
\lambda, & \text{ otherwise}.
\end{cases}
\end{displaymath}

\begin{proposition}\label{prop15}
Let $(S,\mu)$ be a multisemigroup with multiplicities bounded by $\kappa$.
Then $(S,\Phi_{\lambda,\kappa}\circ \mu)$ is a multisemigroup with multiplicities bounded by $\lambda$.
\end{proposition}

\begin{proof}
The axiom \eqref{eq5} in the new situation (for $\lambda$) follows from the axiom \eqref{eq5} in the
old situation (for $\kappa$) by applying the homomorphism $\Phi_{\lambda,\kappa}$ of complete
semirings to both sides. 
\end{proof}

A special case of this construction was mentioned in Subsection~\ref{s4.1}, in that case
$\kappa=\omega$ and $\lambda=1$. A natural question is whether this construction is 
``surjective'' in the sense that any multisemigroups with multiplicities bounded by $\lambda$
can be obtained in this way from a multisemigroups with multiplicities bounded by $\kappa$.
If $\lambda=1$, the answer is {\bf yes} due to the following construction: 

Let $\kappa$ be a nonzero cardinal numbers. Then there is a canonical homomorphism
\begin{displaymath}
\Psi_{\kappa}: \mathrm{Card}_{1}\cong\mathbf{B}\to  \mathrm{Card}_{\kappa}
\end{displaymath}
which sends $0$ to $0$ and also sends $1$ to $\kappa$.
Given a multisemigroup $(S,*)$, we thus may define 
\begin{displaymath}
\mu_{s,t}(r):=
\begin{cases}
0,& r\not\in s*t;\\ 
\kappa, & r\in s*t.
\end{cases}
\end{displaymath}
In other words, we define $\mu$ as the composition of $*$ followed by $\Psi_{\kappa}$. Similarly to the proof
of Proposition~\ref{prop15} we thus get that $(S, \mu)$ is a multisemigroups with multiplicities 
bounded by $\kappa$. As the homomorphism $\Phi_{1,\kappa}\circ \Psi_{\kappa}$ is the 
identity on $\mathbf{B}$, we obtain $(S,*)=(S,\Phi_{1,\kappa}\circ \mu)$.

\subsection{Finitary cardinal reduction}\label{s6.2}

To avoid degenerate examples above, it is natural to rephrase the question as follows: 
Given a multisemigroup $(S,*)$, whether there is a {\em finitary} multi-multisemigroup
$(S,\mu)$ such that $(S,*)=(S,\Phi_{1,\omega}\circ\mu)$. The following example shows that
this is, in general, {\bf not} the case.

\begin{proposition}\label{prop21} 
{\hspace{2mm}}

\begin{enumerate}[$($i$)$]
\item\label{prop21.1} There is a multisemigroup $(\{a,b\},*)$
with the following multiplication table:
\begin{equation}\label{eq11}
\begin{array}{c||c|c}
*& a& b\\
\hline\hline
a&\{a\}&\{a,b\}\\
\hline
b&\{a,b\}&\{a,b\}
\end{array}
\end{equation}
\item\label{prop21.2}
The multisemigroup $(\{a,b\},*)$ is not of the form $(S,\Phi_{1,\omega}\circ \mu)$, for any finitary  
multisemigroup $(S,\mu)$ with multiplicities.
\end{enumerate}
\end{proposition}

\begin{proof}
First we show that the multiplication table \eqref{eq11} really defines a multisemigroup.
We need to check the associativity axiom $x*(y*z)=(x*y)*z$. If $x=y=z=a$, then both sides
are equal to $a$. If $x=b$ or $y=b$ or $z=b$, then both sides are equal to $\{a,b\}$.

Now assume that $(\{a,b\},\mu)$ is a finitary multisemigroup 
with multiplicities. Then $\mu_{a,b}(a)\neq 0$ because $a\in a*b$, moreover, we have $\mu_{a,b}(b)\neq 0$
as $b\in a*b$.

We want to compute $\overline{\mu}_{aab}(a)$ in two different ways, namely, using the decompositions
$(aa)b$ and $a(ab)$.  In the first case, we get 
$\overline{\mu}_{aab}(a)=\mu_{a,a}(a)\mu_{a,b}(a)$. In the second case, we obtain
$\overline{\mu}_{aab}(a)=\mu_{a,a}(a)\mu_{a,b}(a)+\mu_{a,b}(b)\mu_{a,b}(a)$.
As both $\mu_{a,b}(a)\neq 0$ and $\mu_{a,b}(b)\neq 0$, we get a contradiction.
The claim follows.
\end{proof}

\subsection{Deformation of multisemigroups}\label{s6.3}

Let $(S,*)$ be a finite multisemigroup. A finitary multi-multisemigroup  $(S,\mu)$
such that $(S,*)=(S,\Phi_{1,\omega}\circ\mu)$ is called a {\em deformation} of $(S,*)$.
As we saw above, not every finite multisemigroup admits a deformation. It would be interesting
to find some sufficient and necessary conditions for a multisemigroup to admit a non-trivial
deformation. Ideally, it would be really interesting to find some way to describe all possible
deformations of a given multisemigroup. The following is a corollary from the
result in the previous subsection.

\begin{corollary}\label{cor32}
Let $(S,*)$ be a multisemigroup containing two different elements $a$ and $b$ such that 
$a*a=\{a\}$ and $\{a,b\}\subset a*b$ or $\{a,b\}\subset b*a$. Then $(S,*)$ does not admit any
deformation.
\end{corollary}

\begin{proof}
In the case  $\{a,b\}\subset a*b$, the claim follows from the arguments in the proof of 
Proposition~\ref{prop21}. In case $\{a,b\}\subset b*a$ the proof is similar.
\end{proof}

Another obvious observation is the following.

\begin{proposition}\label{prop33}
Let $\cC$ be a finitary $2$-category. Then $(\mathcal{S}(\cC),*)$ admits a deformation. 
\end{proposition}

\begin{proof}
By construction, $(\mathcal{S}(\cC),\mu)$ is a deformation of  $(\mathcal{S}(\cC),*)$.
\end{proof}

\subsection{Connection to twisted semigroup algebras}\label{s6.4}

In case a finite multisemigroup $(S,*)$ is a semigroup, deformations of $(S,*)$ can be understood
as integral twisted semigroup algebras in the sense of \cite{GX}. Indeed, according to the above
definition, a deformation of a finite semigroup $(S,*)$ is given by a map
\begin{displaymath}
\mu:S\times S\to \mathbb{Z}_{\geq 0}, 
\end{displaymath}
which satisfies axiom \eqref{eq5} (the associativity axiom). This is a special case of the 
definition of twisted semigroup algebras, see \cite[Section~3]{GX} or \cite[Equation~(1)]{Wi}.
Typical examples of semigroups which admit non-trivial twisted semigroup algebras
(and hence also non-trivial deformations) are various diagram algebras, see \cite{MMa,Wi} for details.

\section{Multi-multisemigroups and modules over complete semirings}\label{s7}

\subsection{Modules over semirings}\label{s7.1}

Let $R$ be a unital semiring. A {\em (left) $R$-module} is a commutative monoid $(M,+,0)$ together
with the map $\cdot:R\times M\to M$, written $(r,m)\mapsto r\cdot m$, such that
\begin{itemize}
\item $(rs)\cdot m=r\cdot(s\cdot m)$, for all $r,s\in R$ and $m\in M$;
\item $(r+s)\cdot m=r\cdot m +s\cdot m$, for all $r,s\in R$ and $m\in M$;
\item $r\cdot (m+n)=r\cdot m+ r\cdot n$, for all $r\in R$ and $m,n\in M$;
\item $0\cdot m=0$, for all  $m\in M$;
\item $1\cdot m=m$, for all $m\in M$.
\end{itemize}
We refer, for example, to \cite{JM} for more details.
For instance, the multiplication on $R$ defines on $R$ the structure of a left $R$-module
${}_RR$, called the {\em regular} module.

\subsection{Modules over complete semirings}\label{s7.2}

Let $R$ be a complete unital semiring. A {\em (left) complete $R$-module} is an
$R$-module $(M,+,0,\cdot)$ such that 
\begin{itemize}
\item $M$ is complete;
\item $\displaystyle r\cdot \sum_{i\in I}m_i=\sum_{i\in I}r\cdot m_i$, for all $r\in R$ and $m_i\in M$;
\item $\displaystyle \left(\sum_{i\in I}r_i\right)\cdot m=\sum_{i\in I}r_i\cdot m$, for all $r_i\in R$ and $m\in M$.
\end{itemize}
For example, the  regular $R$-module is complete. Another important for us example of a
complete $R$-module is the following.

\begin{example}\label{ex31}
{\rm 
Let $R$ be a complete unital semiring and $X$ a non-empty set. Then the set $R^X$ of all
functions from $X$ to $R$ is a complete abelian group with respect to component-wise
addition, moreover, it has the natural structure of a complete $R$-module given by 
component-wise multiplication with elements in $R$.  This module has, as an 
incomplete submodule, the set of all functions in $R^X$ with finitely many non-zero values.
}
\end{example}

\subsection{Algebras over complete semirings}\label{s7.3}

For a complete unital semiring $R$ and a non-empty set $X$, consider the 
complete $R$-module $R^X$ as in Example~\ref{ex31} above. An {\em algebra} structure on
$R^X$ is a map $\bullet:R^X\times R^X\to R^X$ such that, for all $f_i,f,g,h\in R^X$, we have
\begin{gather}
\label{eqalg1}
\left(\sum_{i\in I}f_i\right)\bullet g=
\sum_{(i,j)\in I\times J} f_i\bullet g;\\
\label{eqalg2}
g\bullet \left(\sum_{i\in I}f_i\right)=
\sum_{(i,j)\in I\times J} g\bullet f_i;\\
\label{eqalg3}
f\bullet (g\bullet h)=(f\bullet g)\bullet h.
\end{gather}
For example, if $X=\{a\}$, then  $R^X=R$ and the multiplication $\cdot$ on $R$ defines
on $R$ the structure of a complete $R$-algebra.

\subsection{Connection to multi-multisemigroups}\label{s7.4}

If $R$ is a semiring and $X$ a set, then, for $x\in X$, we denote by $\chi_x:X\to R$ the 
{\em indicator function} of $x$ defined as follows:
\begin{displaymath}
\chi_x(y)=
\begin{cases}
1,& x=y;\\
0, & x\neq y.
\end{cases}
\end{displaymath}

Our main result in the section is the following:

\begin{theorem}\label{thm31}
{\hspace{2mm}}

\begin{enumerate}[$($i$)$]
\item\label{thm31.1} Let $\kappa$ be a cardinal and $(S,\mu)$ be a multisemigroup with 
multiplicities bounded by $\kappa$. Then $\mathrm{Card}_{\kappa}^S$ has a unique structure
of a complete $\mathrm{Card}_{\kappa}$-algebra such that $\chi_s\bullet\chi_t=\mu_{s,t}$,
for all $s,t\in S$.
\item\label{thm31.2} Conversely, any complete $\mathrm{Card}_{\kappa}^S$-algebra 
$(\mathrm{Card}_{\kappa}^S,\bullet)$ defines a unique structure of 
a multisemigroup with  multiplicities bounded by $\kappa$ on $S$ via
$\mu_{s,t}:=\chi_s\bullet\chi_t$, for $s,t\in S$.
\end{enumerate}
\end{theorem}

\begin{proof}
To prove claim~\eqref{thm31.1}, we first note that uniqueness would follow directly from 
existence as any element in $\mathrm{Card}_{\kappa}^S$ can be written as a
sum, over $S$, of indicator functions. To prove existence, we note that each function can
be {\em uniquely} written as a sum, over $S$, of indicator functions. Therefore, there is
a unique way to extend $\chi_s\bullet\chi_t:=\mu_{s,t}$, for $s,t\in S$, to a map
$\bullet:\mathrm{Card}_{\kappa}^S\times\mathrm{Card}_{\kappa}^S\to\mathrm{Card}_{\kappa}^S$
which satisfies \eqref{eqalg1} and \eqref{eqalg2}. Using \eqref{eqalg1} and \eqref{eqalg2},
it is enough to check the axiom \eqref{eqalg3} for the indicator functions, where it is
equivalent to the axiom \eqref{eq5}, by definition. This proves claim~\eqref{thm31.1}.

To prove claim~\eqref{thm31.2}, we only need to check the axiom \eqref{eq5}. This
axiom follows from the axiom \eqref{eqalg3} applied to the indicator functions.
This completes the proof.
\end{proof}

Theorem~\ref{thm31} suggests that one could define multisemigroups with 
multiplicities from any complete semiring, not necessarily $\mathrm{Card}_{\kappa}$.

\vspace{5mm}

\noindent
Department of Mathematics, Uppsala University, Box. 480,
SE-75106, Uppsala, SWEDEN, email: {\tt love.forsberg\symbol{64}math.uu.se}

\end{document}